\documentclass{amsart}
\usepackage{amsmath}
\usepackage{amsxtra}
\usepackage{amscd}
\usepackage{amsthm}
\usepackage{amsfonts}
\usepackage{amssymb}
\usepackage{eucal} 
\usepackage{fullpage}
\usepackage{url}
\usepackage{setspace}  
\input xy
\xyoption{all}

\newtheorem{thm}{Theorem}
\newtheorem{cor}[thm]{Corollary}
\newtheorem{lem}[thm]{Lemma}
\newtheorem{prop}[thm]{Proposition}

\newtheorem{conj}[thm]{Conjecture}
\newtheorem{falseconj}[thm]{False Claim}

\theoremstyle{remark}
\newtheorem{rem}[thm]{Remark}

\newtheorem{question}[thm]{Question}

\theoremstyle{definition}
\newtheorem{defn}[thm]{Definition}

\newcommand\nc{\newcommand}
\nc\on{\operatorname}
\newcommand\CC{{\mathbb C}}

\newcommand\fp{{\mathfrak p}}
\newcommand\fq{{\mathfrak q}}
\newcommand\fm{{\mathfrak m}}

\newcommand\mcl{\mathcal}
\newcommand\mbb{\mathbb}

\nc\Hom{\on{Hom}}
\nc\Sym{\on{Sym}}
\nc\Spec{\on{Spec}}
\nc\Specm{\on{Specm}}
\nc{\dfp}{\overset{\cdot}{\fp}}
\nc{\dfq}{\overset{\cdot}{\fq}}
\nc{\dfm}{\overset{\cdot}{\fm}}

\begin{document}

\title{Symmetric Powers Do Not Stabilize}
\author{Daniel Litt}

\maketitle

\begin{abstract}
We discuss the stabilization of symmetric products $\on{Sym}^n(X)$ of a smooth projective variety $X$ in the Grothendieck ring of varieties.  For smooth projective surfaces $X$ with non-zero $h^0(X, \omega_X)$, these products do not stabilize; we conditionally show that they do not stabilize in another related sense, in response to a question of R. Vakil and M. Wood \cite{Ravi/Melanie}.  There are analogies between such stabilization, the Dold-Thom theorem, and the analytic class number formula.  Finally, we discuss Hodge-theoretic obstructions to the stabilization of symmetric products, and provide evidence for these obstructions in terms of a relationship between the Newton polygon of a certain ``motivic zeta function" associated to a curve, and its Hodge polygon.
\end{abstract}

\section{Introduction}\label{Introduction}
Let $k$ be a field, and $K_0(\on{Var}_k)$ the Grothendieck ring of varieties over $k$.  This is the free abelian group on isomorphism classes $[X]$ of separated, finite type $k$-schemes (varieties), subject to the following relation:  $$[X]=[Y]+[X\setminus Y] \text{ for $Y\hookrightarrow X$ a closed embedding}.$$
Multiplication is given by $$[X]\cdot [Y]=[X\times Y]$$ on classes of varieties.  This ring was introduced by Grothendieck in 1964 \cite[(Letter of August 16, 1964)]{Groth} in a letter to Serre.  Let $\mbb{L}:=[\mbb{A}^1_k]$ be the class of the affine line in this ring.

Ravi Vakil and Melanie Wood have conjectured \cite[Conjecture 1.25]{Ravi/Melanie} that in a certain completion of $K_0(Var_k)[\mbb{L}^{-1}]$, denoted $\hat K$ (to be defined in Section 2), the limit $$\lim_{n\to \infty} \frac{[\on{Sym}^n X]}{\mbb{L}^{n\on{dim}(X)}}$$ exists for connected $X$.  They call the existence of this limit ``motivic stabilization of symmetric powers" or MSSP for short; they show this conjecture is true in many cases.  The main goal of this note is to provide some evidence that it is false in general.

We introduce analogous but more accessible claims (False Claims \ref{False Claim 1} and \ref{False Claim 2}) which hold true in every case where MSSP is known to hold.  Furthermore, we show unconditionally  that both False Claims are untrue in general for $k=\mbb{C}$---counterexamples include smooth projective surfaces $X$ with geometric genus $p_g(X)\not=0$ (Corollary \ref{falseclaimsarefalsecor}).\footnote{While this paper was in preparation, Melanie Wood provided several other counterexamples to False Claim \ref{False Claim 1} \cite{Melaniepersonal}.  In particular, her arguments combined with those here show that for smooth projective surfaces with a non-vanishing even plurigenus, False Claim \ref{False Claim 1} fails; in addition to the results here, this covers e.g. Enriques surfaces and certain surfaces of general type but with $p_g=0$.}   We show that these counterexamples are also counterexamples to MSSP, conditional on the truth of either of two well-known conjectures about $K_0(\on{Var}_k)$ (the ``Cut-and-Paste" conjecture of Liu and Sebag \cite{Cut-and-paste} or the conjecture that $\mathbb{L}$ is not a zero divisor), and thus MSSP is false if either of these conjectures are true (Corollary \ref{msspfalsecor}).

Finally, we propose Hodge-theoretic heuristics explaining the failure of MSSP, and give some evidence for these heuristics via explicit computations for curves.  In particular, we prove a motivic analogue of the classical theorem that ``the Newton polygon lies above the Hodge polygon," for Weil zeta functions associated to varieties.  Namely, if $X$ is a smooth projective curve with a rational point, we show that over a general field, a certain ``motivic Newton polygon" associated to $X$ lies above the Hodge polygon of $X$; this implies the classical result for Weil zeta functions if $X$ is defined over a finite field (Lemma \ref{newtonabovehodgelemma}).  Over algebraically closed fields of characteristic zero, we show that the Newton and Hodge polygons of curves are equal (Corollary \ref{newtonequalshodgecor}).  These last results are of independent interest, and are contained in Section 5, which can be read independently of Sections $3$ and $4$.

\subsection{Acknowledgments}  
This note owes a great deal to conversations with David Ayala, Rebecca Bellovin, Jeremy Booher, Jonathan Campbell, Brian Conrad, Fran\c{c}ois Greer, Sander Kupers, Sam Lichtenstein, Cary Malkiewich, Emmy Murphy, Mircea Musta\c{t}\u{a}, Matthew Satriano, Arnav Tripathy, Ravi Vakil, Kirsten Wickelgren, and Melanie Wood.

\section{False Claims and Motivation}
Let us first give the statement of MSSP.  We define a filtration on $K_0(\on{Var}_k)[\mathbb{L}^{-1}]$, given by dimension.  Let $F^iK_0(\on{Var}_k)[\mathbb{L}^{-1}]\subset K_0(\on{Var}_k)[\mathbb{L}^{-1}]$ be the (additive) subgroup generated by elements of the form $[X]/\mathbb{L}^r$ where $X$ is of pure dimension $\leq i+r$.  

We define the ring $\hat K$ to be the completion of $K_0(\on{Var}_k)$ at this filtration, that is, $$\hat K:=\varprojlim K_0(\on{Var}_k)[\mathbb{L}^{-1}]/F^iK_0(\on{Var}_k)[\mathbb{L}^{-1}].$$  The ring $\hat K$ was initially defined by Kontsevich \cite{kontsevich} as the ring in which the values of motivic integrals lie.  The MSSP Conjecture takes place in $\hat K$:
\begin{conj}[Vakil, Wood; motivic stabilization of symmetric powers, or MSSP {\cite[Conjecture 1.25]{Ravi/Melanie}}]  \label{msspconj}
Let $X$ be a connected variety over $k$.  Then $$\lim_{n\to \infty} \frac{[\on{Sym}^n X]}{\mathbb{L}^{n\dim(X)}}$$ exists in $\hat K$.
\end{conj}

\begin{rem} \label{smoothgeneration}
Note that resolution of singularities implies that $F^iK_0(\on{Var}_k)$ is (additively) generated by elements of the form $[X]/\mathbb{L}^r$ with $X$ \emph{smooth and proper} and of pure dimension $\leq i+r$.  In arbitrary characteristic, $F^iK_0(\on{Var}_k)$ is generated by elements of the form $[X]/\mathbb{L}^r$ with $X$ \emph{proper} and of pure dimension $\leq i+r$.  To see this, note that by Noetherian induction, $K_0(\on{Var}_k)$ is generated by the classes of affine varieties.  But affine varieties are in the subring generated by proper varieties, as they may be compactified, e.g. by considering their scheme-theoretic image under some quasi-projective embedding.
\end{rem}

Unfortunately, it is difficult to understand the behavior of $\hat K$, as it is unknown whether or not $\mbb{L}$ is a zero divisor.  Thus, for much of this note, we will instead consider the completion $R$ of $K_0(\on{Var}_k)$ at the ideal $(\mbb{L})$.  That is, $$R:=\varprojlim K_0(\on{Var}_k)/\mbb{L}^n.$$  

\begin{rem} \label{separatedtopology}
It is not known whether $$\bigcap_n (\mathbb{L}^n)=(0).$$  That is, it is unclear whether the $\mbb{L}$-adic topology on $K_0(\on{Var}_k)$ is separated. 
\end{rem}

The completion $R$ has three advantages:  
\begin{itemize}
\item Many well-known homomorphisms from $K_0(\on{Var}_k)$ to other rings (so-called ``motivic measures" \cite{motivicmeasures}) extend continuously to $R$, but not to $\hat K$.  For example, if $k=\mathbb{F}_q$, with $q=p^n$, the homomorphism $$\psi_q: K_0(\on{Var}_k)\to \mathbb{Z}$$ defined by $$\psi_q:  [X]\mapsto \#X(\mbb{F}_q)$$ extends to a continuous homomorphism $R\to \mathbb{Z}_p$.
\item $R$ is easier to work with than $\hat K$---in particular, Theorem \ref{zsbthm} below shows that $K_0(\on{Var}_k)/\mathbb{L}$ remembers exactly the stable birational geometry of smooth projective varieties, so convergence of limits in $R$ ``to first order" has geometric meaning.
\item If $k$ is algebraically closed of characteristic zero, we may define a surjection of topological rings $$\mathbb{D}: R[\mathbb{L}^{-1}]\to \hat K$$ (see Remark \ref{completions}); if $\mathbb{L}$ is not a zero-divisor, then $\mbb{D}$ is an isomorphism.  For a smooth projective curve $X$, $$\mbb{D}([\on{Sym}^n(X)])=\frac{[\on{Sym}^n(X)]}{\mathbb{L}^n}.$$  For a smooth projective surface $S$, $$\mbb{D}([\on{Hilb}^n(S)])= \frac{[\on{Sym}^n(S)]}{\mathbb{L}^{2n}} \bmod F^{-1}K_0(\on{Var}_k)[\mbb{L}^{-1}].$$ 
\end{itemize}

 Thus, we consider the following False Claim as an alternative to Conjecture \ref{msspconj}---much of the work in this note will be aimed at examining the circumstances in which this claim fails. 

\begin{falseconj}\label{False Claim 1}
For connected $X$, the limit $$\lim_{n\to \infty} [\on{Sym}^n X]$$ exists in $R$.
\end{falseconj}

False Claim \ref{False Claim 1} would immediately imply
\begin{falseconj}\label{False Claim 2}
For $X$ connected and $n, m\gg 0$, $[\on{Sym}^n X]=[\Sym^m X]$ in $K_0(\on{Var}_k)/\mbb{L}$.
\end{falseconj} 
which we will disprove in Section 4, for $k$ an algebraically closed field of characteristic $0$.  

Before disproving these claims, however, we would like to take the odd step of motivating our False Claims \ref{False Claim 1} and \ref{False Claim 2}, as well as the MSSP conjecture of Vakil and Wood.  In particular, we will make a case that these false claims are natural, despite their falsehood.

First, let $C$ be a smooth proper genus $g$ curve over a field $k$, with a $k$-rational point.  Then for $n\gg 0$, $\on{Sym}^n C$ is a (Zariski) $\mathbb{P}^{n-g}$-bundle over $\on{Jac}(C)$, and so $$[\on{Sym}^n C]=(1+\mathbb{L}+\cdots +\mathbb{L}^{n-g})[\on{Jac}(C)]$$ in $K_0(\on{Var}_k)$.  Thus the limit 
$$\lim_{n\to \infty} [\on{Sym^n} C]$$ clearly exists in $R$ and equals 
\begin{equation}\label{dold-thom}
[\on{Jac}(C)](1+\mbb{L}+\mbb{L}^2+\cdots).
\end{equation}

So False Claim \ref{False Claim 1} holds for smooth proper curves with a rational point. 

Similarly, False Claim \ref{False Claim 1} holds for connected rational or uniruled surfaces, and for connected varieties $X$ whose classes $[X]\in K_0(\on{Var}_k)$ are polynomial in $\mathbb{L}$, e.g. (split) affine algebraic groups and their homogeneous spaces.  The proofs of these claims are not difficult, so we omit them.  These examples also provide evidence for MSSP; indeed, every $X$ for which MSSP is known to hold also satisfies False Claim \ref{False Claim 1}, and vice versa.

MSSP and False Claim \ref{False Claim 1}, when true, provide an algebro-geometric analogue of the following beautiful theorem of Dold and Thom.
\begin{thm}[Dold-Thom \cite{doldthom}] \label{doldthomthm}
Let $X$ be a connected CW complex with basepoint $x_0$.  Let $SP^\infty(X)$ be the direct limit of the spaces $\on{Sym}^n X$, under the maps $\on{Sym}^n X\hookrightarrow \on{Sym}^{n+1} X$ given by $(x_1, x_2, \cdots, x_n)\mapsto (x_0, x_1, ..., x_n)$.  Then there is an identification $$SP^\infty(X)\simeq_w \prod_{i>0}K(H_i(X, \mathbb{Z}), i),$$
where $K(G, n)$ is an Eilenberg-Maclane space; namely, a space with $\pi_n(K(G, n))=G, \pi_m(K(G, n))=\{0\}$ for $m\not=n$.  The symbol $\simeq_w$ denotes weak equivalence.
\end{thm}

Let us compare this to the situation in False Claim \ref{False Claim 1} for $X$ a smooth proper curve of genus $g$ over $\mathbb{C}$.  In this case, $\on{Jac}(X)$ has the homotopy type of a $K(H_1(X, \mathbb{Z}), 1)$ by construction, and $H_2(X, \mathbb{Z})=\mathbb{Z}$, so $K(H_2(X, \mathbb{Z}), 2)\simeq_w \CC\mathbb{P}^\infty\simeq \varinjlim \CC\mathbb{P}^n$, which we may view as being represented by the class $$[\mathbb{CP}^\infty]:=\lim_{n\to \infty} [\mathbb{P}^n]=1+\mathbb{L}+\mathbb{L}^2+\cdots$$ in $R$.  Thus the Dold-Thom theorem gives $$\varinjlim \on{Sym}^n X\simeq_w K(\mathbb{Z}^{2g}, 1)\times K(\mathbb{Z},2)\simeq_w \on{Jac}(X)\times \mathbb{CP}^\infty.$$  Likewise our computation (\ref{dold-thom}) above gives
$$\lim_{n\to \infty} [\on{Sym}^n X]=[\on{Jac}(C)][\mathbb{CP}^\infty]$$ in $R$.  So we may view the Dold-Thom theorem as evidence for False Claim \ref{False Claim 1}.

Finally, consider the case $k=\mathbb{F}_q$, where $q=p^n$.  In this case, there is a ``point-counting" ring homomorphism $\psi_q: K_0(\on{Var}_k)\to \mathbb{Z}$ sending a variety $X$ to $\#X(\mathbb{F}_q)$.  By analogy to the zeta function $$\zeta_X(t):=\on{exp}\left(\sum_n \frac{\#X(\mathbb{F}_{q^n})}{n}t^n\right)$$ appearing in the Weil conjectures, Kapranov \cite[1.3]{kapranov} introduced the following ``motivic zeta function," associated to a variety $X$ over an arbitrary field:  $$Z^{mot}_X(t)=\sum_{i=0}^\infty [\on{Sym}^n X]t^n\in K_0(\on{Var}_k)[[t]].$$  In the case that $k$ is finite, $$\psi_q(Z^{mot}_X(t))=\zeta_X(t)$$ where $\psi_q$ denotes the homomorphism $K_0(\on{Var}_k)[[t]]\to \mathbb{Z}[[t]]$ induced by applying the point-counting map coefficient-wise.  

We have $$\lim_{n\to \infty}[\on{Sym}^n X]= (1-t)Z_X^{mot}(t)|_{t=1}$$ formally, where both expressions are evaluated in $R$; if $X$ is a curve or a rational or uniruled surface, $(1-t)Z_X^{mot}(t)$ is everywhere convergent as a power series over $R$.  If $X$ is a curve, the limit on the left specializes under $\psi_q$ to the ``analytic class number formula" for the zeta functions appearing in the Weil conjectures; we have that 
\begin{equation}\label{classnumber1}
\on{res}_{t=1}\zeta_X(t)=\frac{\#\on{Jac}(X)(\mathbb{F}_q)}{1-q}
\end{equation} and likewise 
\begin{equation}\label{classnumber2}
(1-t)Z_X^{mot}(t)|_{t=1}=\frac{[\on{Jac}(X)]}{1-\mathbb{L}}.
\end{equation}
  Thus we may view False Claim \ref{False Claim 1} or MSSP as analogues of analytic class number formulas.  To put it another way, these claims are analogues of the fact that the zeta functions appearing in the Weil conjectures have a pole of order 1 at $t=1$---\emph{assuming the power series expansion for $(1-t)Z^{mot}_X(t)$ is valid at $t=1$}.  We will expand on this last heuristic in Section 5.

\begin{rem} \label{analyticclassnumber}
We refer to Equations (\ref{classnumber1}) and (\ref{classnumber2}) as ``analytic class number formulas" because of their (informal) resemblance to the analytic class number formula for Dedekind zeta functions associated to number fields.
\end{rem}

\begin{rem} \label{zetarationality}
By analogy to the Weil conjectures, one might guess that $Z_X^{mot}(t)$ is the power series associated to a rational function with coefficients in $K_0(\on{Var}_k)$.  Kapranov shows that this is true for curves with a rational point \cite[(1.3.5)(a)]{kapranov}, where the hypothesis of the existence of a rational point is left implicit.  For curves with no rational point the argument does not work.  The issue is that the usual Picard functor is not representable in this case, and so $\on{Sym}^n(X)$ is not a projective space bundle over $\on{Pic}^n(X)$, which is an obstruction to Kapranov's argument.  It is unclear to the author if this issue can be rectified.  On the other hand, Larsen and Lunts have shown \cite{motivicmeasures, rationalitycriteria} that these zeta functions are not rational over $K_0(\on{Var}_k)$ for most surfaces.  The question of the rationality of $Z^{mot}_X(t)$ over $K_0(\on{Var}_k)[\mbb{L}^{-1}]$ is open.
\end{rem}

The plan for the rest of this note is as follows.  In Section 3 we will introduce several facts and conjectures about $K_0(\on{Var}_k)$ and discuss their interplay.  In Section 4 we will disprove False Claims \ref{False Claim 1} and \ref{False Claim 2} and deduce the conditional falsity of MSSP in the case of smooth projective surfaces $X$ with non-vanishing $H^0(X, \omega_X)$.  In Section 5 we will propose Hodge-theoretic heuristics for the failure of MSSP and give evidence for them in terms of the ``Newton polygons" of Kapranov zeta functions $Z_X^{mot}(t)$ of curves.
\section{Preliminaries and Discussion}
From here on, unless otherwise stated, $k$ will be algebraically closed of characteristic zero.

In \cite{Bittner}, Bittner gives the following useful presentation of $K_0(\on{Var}_k)$---the proof uses resolution of singularities and weak factorization of rational maps.  We will use her description of $K_0(\on{Var}_k)$ and some of her constructions to relate MSSP to False Claims \ref{False Claim 1} and \ref{False Claim 2}.
\begin{thm}[Bittner {\cite[Theorem 3.1]{Bittner}}] \label{bittnerthm}
$K_0(\on{Var}_k)$, for $k$ algebraically closed and of characteristic zero, is generated by the classes of smooth proper $k$-varieties, subject only to the following relations:  $$[\on{Bl}_{Y}(X)]-[E]=[X]-[Y]$$ for $X$ proper, $Y$ a smooth closed subvariety of $X$, and $E$ the exceptional divisor of the blowup $\on{Bl}_Y(X)$.
\end{thm}

Bittner uses this presentation to construct a ``duality map" $\mbb{D}: K_0(\on{Var}_k)\to K_0(\on{Var}_k)[\mbb{L}^{-1}]$, which we will use heavily.
\begin{cor}[Duality map {\cite[Corollary 3.4]{Bittner}}] \label{dualitycor}
There exists a unique ring homomorphism $\mbb{D}: K_0(\on{Var}_k)\to K_0(\on{Var}_k)[\mbb{L}^{-1}]$ with $$\mbb{D}([X])=\frac{[X]}{\mbb{L}^{\on{dim}(X)}}$$ for $X$ smooth and proper.
\end{cor}
\begin{proof}
By Theorem \ref{bittnerthm}, $\mbb{D}$ is uniquely determined by its value on smooth proper varieties; we must check that $$\mbb{D}([\on{Bl}_Y(X)])-\mbb{D}([E])=\mbb{D}([X])-\mbb{D}([Y])$$ in $K_0(\on{Var}_k)[\mbb{L}^{-1}]$ for $X, Y$ as in Theorem \ref{bittnerthm}.  Using the fact that
$$(1-\mbb{L})[\mbb{P}^n]=1-\mbb{L}^{n+1}$$
 we have
\begin{align*}
\mbb{L}^{\on{dim}(X)}(\mbb{D}([\on{Bl}_Y(X)])-\mbb{D}([E])) &= [\on{Bl}_Y(X)]-\mbb{L}[E]\\
&=[X]-[Y]+[E]-\mbb{L}[E]\\
&=[X]-[Y]+(1-\mbb{L})[\mbb{P}^{\on{dim}(X)-\on{dim}(Y)-1}][Y]\\
&=[X]-[Y]+(1-\mbb{L}^{\on{dim}(X)-\on{dim}(Y)})[Y]\\
&=[X]-\mbb{L}^{\on{dim}(X)-\on{dim}(Y)}[Y]\\
&=\mbb{L}^{\on{dim}(X)}(\mbb{D}([X])-\mbb{D}([Y])).
\end{align*}
Dividing by $\mbb{L}^{\on{dim}(X)}$ gives the claim.  That $\mbb{D}$ is a ring homomorphism follows from the additivity of dimension.
\end{proof}
Note that $\mbb{D}(\mbb{L})=\mbb{D}(\mbb{P}^1)-\mbb{D}([\{\on{pt}\}])=\mbb{L}^{-1}$.  Thus $\mbb{D}$ induces a map $\mbb{D}': K_0(\on{Var}_k)[\mbb{L}^{-1}]\to K_0(\on{Var}_k)[\mbb{L}^{-1}]$, satisfying $\mbb{D}'\circ \mbb{D}'=\on{id}$.

We will need one further result on $K_0(\on{Var}_k)$, due to Larsen and Lunts \cite{motivicmeasures}; it also follows from Bittner's presentation of $K_0(\on{Var}_k)$.  
\begin{defn}[Stable birationality]
Let $X$ and $Y$ be two varieties; recall that $X$ is \emph{stably birational} to $Y$ if $X\times \mbb{A}^n$ is birational to $Y\times \mbb{A}^m$ for some $m, n$.  If $k$ is separably closed, $SB$ denotes the monoid of stable birational equivalence classes of smooth, connected, proper $k$-varieties under the operation of Cartesian product.
\end{defn}

\begin{thm}[Larsen and Lunts, {\cite[Proposition 2.7]{motivicmeasures}}] \label{zsbthm}
Let $k$ be algebraically closed of characteristic zero.  Then there is a ring homomorphism $sb: K_0(\on{Var}_k)/\mbb{L}\to \mbb{Z}[SB]$, sending the class of a smooth proper variety to its stable birational equivalence class; $sb$ is an isomorphism.
\end{thm}

For future reference, we will list here two conjectures about $K_0(\on{Var}_k)$; we will show that MSSP fails conditional on the truth of either of these conjectures.

\begin{conj}[Cut-and-paste conjecture, Liu and Sebag \cite{Cut-and-paste}] \label{cutandpasteconj}
Let $X$ and $Y$ be varieties over $k$.  If $[X]=[Y]$ in $K(\on{Var}_k)$, then there exist disjoint locally closed subvarieties $X_i$ of $X$, $Y_i$ of $Y$, such that $$X=\cup X_i\text{ and } Y=\cup Y_i$$ and $$X_i\simeq Y_i$$ for all $i$.
\end{conj}

\begin{conj}[Cancellation of the Lefschetz motive, {\cite[3.3]{denefloeser}}, {\cite[remarks after Assertion 1]{Cut-and-paste}}] \label{lisnotazerodivisorconj}
$\mbb{L}$ is not a zero divisor in $K(\on{Var}_k)$.
\end{conj}

This latter conjecture is a common assumption for those working with the Grothendieck ring of varieties, e.g.  in \cite[Remark 16]{Cut-and-paste}, \cite[7.1]{lamy/sebag}.

We also record consequences of these conjectures which will be required later.

\begin{prop}[Stable birationality in $\hat K$] \label{stablebirathatkprop}
Let $X$ and $Y$ be irreducible varieties over $k$.  Then if the Cut-and-Paste Conjecture (Conjecture \ref{cutandpasteconj}) holds for $k$, we have that $$\frac{[X]}{\mbb{L}^{\on{dim}(X)}}=\frac{[Y]}{\mbb{L}^{\on{dim}(Y)}}$$ in $F^0K_0(\on{Var}_k)[\mbb{L}^{-1}]/F^{-1}K_0(\on{Var}_k)[\mbb{L}^{-1}]$ only if $X$ is stably birational to $Y$.
\end{prop}
\begin{proof}
The idea of the argument is to translate this equality into an equality in $K_0(\on{Var}_k)$, and then apply the cut-and-paste conjecture there.

If $$\frac{[X]}{\mbb{L}^{\on{dim}(X)}}=\frac{[Y]}{\mbb{L}^{\on{dim}(Y)}}$$ in $F^0K_0(\on{Var}_k)[\mbb{L}^{-1}]/F^{-1}K_0(\on{Var}_k)[\mbb{L}^{-1}]$, there are equidimensional varieties $S_i$ such that $$\frac{[X]}{\mbb{L}^{\on{dim}(X)}}-\frac{[Y]}{\mbb{L}^{\on{dim}(Y)}}=\sum_{i\in I} a_i\frac{[S_i]}{\mbb{L}^{n_i}}$$ in $K_0(\on{Var}_k)[\mbb{L}^{-1}]$, where $n_i>\on{dim}(S_i), a_i\in \mathbb{Z}$, and $I$ is a finite index set.  Equivalently, for some $N\gg 0$, $$\mbb{L}^{N-\on{dim}(X)}[X]-\mbb{L}^{N-\on{dim}(Y)}[Y]=\sum_{i\in I} a_i\mbb{L}^{N-n_i}[S_i]$$ in $K_0(\on{Var}_k).$  Rearranging terms, we have for $J\subset I, J:=\{i\in I| a_i< 0\}$ that 
$$\mbb{L}^{N-\on{dim}(X)}[X]+\sum_{j\in J}|a_j|\mbb{L}^{N-n_j}[S_j]=\mbb{L}^{N-\on{dim}(Y)}[Y]+\sum_{i\in I\setminus J} |a_i|\mbb{L}^{N-n_i}[S_i].$$

Now, if the cut-and-paste conjecture holds, the equality above implies that we may write $$X'=(\mbb{A}^{N-\on{dim}(X)}\times X)\cup \bigcup_{j\in J}\mbb{A}^{N-n_j}\times (\underbrace{S_j\cup \cdots \cup S_j}_{|a_j| \text{ times}})$$
and 
$$Y'=(\mbb{A}^{N-\on{dim}(Y)}\times Y)\cup \bigcup_{i\in I\setminus J} \mbb{A}^{N-n_i}\times (\underbrace{S_i\cup \cdots \cup S_i}_{|a_i| \text{ times}})$$
as disjoint unions of isomorphic locally closed subsets.  That is, $X'$ and $Y'$ are equidecomposable.  But $X', Y'$ have exactly one connected component of dimension $N$---respectively, $\mbb{A}^{N-\on{dim}(X)}\times X$ and $\mbb{A}^{N-\on{dim}(Y)}\times Y$---and all other connected components have dimension less than $N$, as $n_i>\on{dim}(S_i)$.  So $\mbb{A}^{N-\on{dim}(X)}\times X, \mbb{A}^{N-\on{dim}(Y)}\times Y$ must be birational.  Thus $X$ and $Y$ are stably birational.
\end{proof}

We also have a similar result contingent on the truth of Conjecture \ref{lisnotazerodivisorconj}.
\begin{prop} \label{assgradprop}
Suppose $\mbb{L}$ is not a zero divisor in $K_0(\on{Var}_k)$ (Conjecture \ref{lisnotazerodivisorconj}).  Then $$F^0K_0(\on{Var}_k)[\mbb{L}^{-1}]/F^{-1}K_0(\on{Var}_k)[\mbb{L}^{-1}]\simeq K_0(\on{Var}_k)/\mbb{L}.$$
\end{prop}
\begin{proof}
We first define a map $f:  K_0(\on{Var}_k)/\mbb{L}\to F^0K_0(\on{Var}_k)[\mbb{L}^{-1}]/F^{-1}K_0(\on{Var}_k)[\mbb{L}^{-1}]$.  This map is induced by $\mbb{D}$; we need to show that the ideal $(\mbb{L})$ maps into $F^{-1}K_0(\on{Var}_k)[\mbb{L}^{-1}]$, and that the image of $\mbb{D}$ is contained in $F^0K_0(\on{Var}_k)[\mbb{L}^{-1}]$.  For the former statement, note that the ideal $(\mbb{L})$ is generated by elements of the form $\mbb{L}[X]$, where $X$ is smooth, connected, and proper.  Then $$\mbb{D}(\mbb{L}[X])=\frac{[X]}{\mbb{L}^{\on{dim}(X)+1}}\in F^{-1}K_0(\on{Var}_k)[\mbb{L}^{-1}]$$ as desired.  The latter statement follows analogously (indeed, $\mbb{D}((\mbb{L}^n))$ equals $F^{-n}K_0(\on{Var}_k)[\mbb{L}^{-1}]$ for all $n\geq 0$).

If $\mbb{L}$ is not a zero divisor, we may define an inverse map.  By the previous observation, $\mbb{D}'(F^0K_0(\on{Var}_k)[\mbb{L}^{-1}])$ is exactly the image of $K_0(\on{Var}_k)$ in $K_0(\on{Var}_k)[\mbb{L}^{-1}]$ via the natural inclusion; as $\mbb{L}$ is not a zero divisor, this image is isomorphic to $K_0(\on{Var}_k)$ itself.  So $\mbb{D}'$ induces a map $g: F^0K_0(\on{Var}_k)[\mbb{L}^{-1}]\to K_0(\on{Var}_k)$.  We need to check that $F^{-1}K_0(\on{Var}_k)$ maps to $(\mbb{L})$.  But the verification proceeds as in the previous paragraph.
\end{proof}
\begin{cor} \label{zsbcor}
If $\mbb{L}$ is not a zero divisor in $K_0(\on{Var}_k),$ $F^0K_0(\on{Var}_k)[\mbb{L}^{-1}]/F^{-1}K_0(\on{Var}_k)[\mbb{L}^{-1}]\simeq \mbb{Z}[SB]$.  If $X$ is smooth and proper, this isomorphism sends $[X]/\mbb{L}^{\on{dim}(X)}$ to the stable birational equivalence class of $X$.
\end{cor}
\begin{proof}
This follows immediately from Theorem \ref{zsbthm} and Proposition \ref{assgradprop}.
\end{proof}
\begin{rem}\label{completions}
Note that $\mathbb{D}$ extends to a continuous surjection $\mathbb{D}: R[\mathbb{L}^{-1}]\to\hat K$; if $\mathbb{L}$ is not a zero divisor, the methods of Proposition \ref{assgradprop} show that this is an isomorphism.  As claimed in Section 2, for a smooth projective curve $X$, $$\mbb{D}([\on{Sym}^n(X)])=\frac{[\on{Sym}^n(X)]}{\mathbb{L}^n}.$$  For a smooth projective surface $S$, $$\mbb{D}([\on{Hilb}^n(S)])= \frac{[\on{Sym}^n(S)]}{\mathbb{L}^{2n}} \bmod F^{-1}K_0(\on{Var}_k)[\mbb{L}^{-1}],$$
where we use that $\on{Hilb}^n(S)$ is smooth and projective for $S$ smooth and projective \cite[p. 167]{FGA explained}.
\end{rem}
\section{Stable Birationality Of Symmetric Powers}
The main geometric content of this section is the following:
\begin{thm} \label{msspfailsthm}
Let $X$ be a smooth connected projective surface with $h^0(X, \omega_X)\not=0$.  Let $m$ be a non-negative integer.  Then for all sufficiently large $n$, $\on{Sym}^n(X)$ is not stably birational to $\on{Sym}^m(X)$.
\end{thm} 

The idea of the proof is to produce a moving family of unirational subvarieties of $\on{Sym}^n(X)$ of high dimension, contradicting the following theorem of Mumford.

\begin{thm}[Mumford, \cite{mumford}, corollary on page 203] \label{mumfordthm}
There exists a codimension-one subvariety $W$ of $\on{Sym}^n(X)$ so that if $Y\subset \on{Sym}^n(X)\setminus W$ consists entirely of rationally equivalent $0$-cycles, then $Y$ has dimension at most $n$.
\end{thm}
Mumford's proof uses the so-called ``symplectic argument" \cite[Chapter 10]{voisin}.  The idea is to show that a connected subvariety of $\on{Sym}^n(X)$ consisting of rationally equivalent zero-cycles and containing a generic $0$-cycle must lie tangent to an isotropic subspace of any two-form on $\on{Sym}^n(X)$.  One can construct generically non-degenerate two-forms on the smooth locus of $\on{Sym}^n(X)$ given any non-zero two-form on $X$, giving an upper bound on the dimension of most such varieties.
\begin{proof}[Proof of Theorem \ref{msspfailsthm}]
Without loss of generality $n>m$.  Assume for the sake of contradiction that $\on{Sym}^n(X)\times \mbb{A}^l$ is birational to $\on{Sym}^m(X)\times \mbb{A}^{2n-2m+l}$ for some $l$.  Then there exists a variety $U$ which may be embedded as a dense open subvariety of both $\on{Sym}^n(X)\times \mbb{A}^l$ and $\on{Sym}^m(X)\times \mbb{A}^{2n-2m+l}.$

A non-empty fiber of the projection map $\pi_m: U\to \on{Sym}^m(X)$ is a dense open subset of  $\mbb{A}^{2n-2m+l}$, and is thus of dimension $2n-2m+l$.  Choosing any $k$-point $x$ in the image of the \emph{other} projection map $\pi_n: U\to \on{Sym}^n(X)$ there exists some $k$-point $y\in \on{Sym}^m(X)$ so that $x\in \pi_n(\pi_m^{-1}(y))$.

Choosing $x\in \on{Sym}^n(X)$ lying away from the subvariety $W$ of $\on{Sym}^n(X)$ coming from Theorem \ref{mumfordthm}, we choose $y$ as above and let $Y=\pi_n(\pi_m^{-1}(y))$.  As $\pi_m^{-1}(y)$ is an open subset of affine space, $Y$ is unirational; furthermore $Y$ has dimension $2n-2m$, as the non-empty fibers of $\pi_n$ have dimension $l$.  As $Y$ is unirational, points in it correspond to rationally equivalent $0$-cycles.  So for $n>2m$, $Y\setminus (Y\cap W)$ is a subvariety of $\on{Sym}^n(X)\setminus W$ consisting of rationally equivalent $0$-cycles, of dimension larger than $n$, which contradicts Theorem \ref{mumfordthm}.
\end{proof}

\begin{cor}[False Claims \ref{False Claim 1} and \ref{False Claim 2} are false] \label{falseclaimsarefalsecor}
Let $X$ be as in Theorem \ref{msspfailsthm}.  Then the classes $[\on{Sym}^n(X)]$ do not stabilize in $K_0(\on{Var}_k)/\mbb{L}$.
\end{cor}
Before proving this Corollary, we will need the following result of G\"ottsche.
\begin{thm}[G\"ottsche, {\cite[Theorem 1.1]{gottsche}}] \label{gottschethm}
Let $S$ be a smooth projective surface over $k$.  Then in $K_0(\on{Var}_k)$, we have $$[\on{Hilb}^n(S)]=\sum_{(1^{\alpha_1}, ..., n^{\alpha_n})\in P(n)} \mbb{L}^{n-|\alpha|}\prod_{i=1}^n[\on{Sym}^{\alpha_i}(S)].$$ Here $\on{Hilb}^n(X)$ the Hilbert scheme of length $n$ subschemes of $X$ and $P(n)$ is the set of partitions of $n$.  If $\alpha\in P(n)$ we write $\alpha=(1^{\alpha_1}, 2^{\alpha_2}, \cdots, n^{\alpha_n})$, and define $|\alpha|=\sum_i \alpha_i$.  
\end{thm}
An immediate consequence is that $$[\on{Hilb}^n(S)]=[\on{Sym}^n(S)] \bmod \mbb{L}.$$
\begin{proof}[Proof of Corollary \ref{falseclaimsarefalsecor}]
As $\on{Hilb}^n(X)$ is birational to $\on{Sym}^n(X)$ \cite[p. 161]{FGA explained} via the Hilbert-Chow morphism, we have from Theorem \ref{msspfailsthm} that $\on{Hilb}^n(X)$ is not stably birational to $\on{Hilb}^m(X)$ for $n\gg m$.  Furthermore, as $X$ is a smooth projective surface, $\on{Hilb}^n(X)$ is smooth and projective \cite[p. 167]{FGA explained}, so $[\on{Hilb}^n(X)]$ maps to its stable birational equivalence class via the isomorphism $sb: K_0(\on{Var}_k)/\mbb{L}\to \mbb{Z}[SB]$.  But the images of the $\on{Hilb}^n(X)$ in $\mbb{Z}[SB]$ do not stablize, so the classes $[\on{Hilb}^n(X)]$ in $K_0(\on{Var}_k)/\mbb{L}$ do not stabilize either.  Finally, $$[\on{Hilb}^n(X)]=[\on{Sym}^n(X)] \bmod \mbb{L},$$ so the classes $[\on{Sym}^n(X)]$ do not stabilize, as desired.  
\end{proof}
\begin{cor}[Conditional falsity of MSSP] \label{msspfalsecor}
If either (a) Conjecture \ref{cutandpasteconj} (Cut-and-Paste) or (b) Conjecture \ref{lisnotazerodivisorconj} ($\mbb{L}$ is not a zero divisor) hold for $k$, then MSSP fails for $X$ as in Theorem \ref{msspfailsthm}.
\end{cor}
\begin{proof}
The fact that (a) Conjecture \ref{cutandpasteconj} implies the failure of MSSP follows immediately from Theorem \ref{msspfailsthm} and Proposition \ref{stablebirathatkprop}.  For (b), note that $$\frac{[\on{Hilb}^n(X)]}{\mbb{L}^{2n}}=\frac{[\on{Sym}^n(X)]}{\mbb{L}^{2n}}$$ in $$F^0K_0(\on{Var}_k)[\mbb{L}^{-1}]/F^1K_0(\on{Var}_k)[\mbb{L}^{-1}]$$ as $\on{Sym}^n(X)$ and $\on{Hilb}^{n}(X)$ are birational.  But $\on{Hilb}^n(X)$ is smooth and proper, so its image in $\mbb{Z}[SB]$ via the map defined in Proposition \ref{assgradprop} and Corollary \ref{zsbcor} is its stable birational equivalence class.  These images do not stabilize, by Theorem \ref{msspfailsthm}.
\end{proof}
\section{Heuristics for Failure of MSSP, and Newton Polygons}
In this final section, we will describe a heuristic predicting for which $X$ False Claim \ref{False Claim 2} (and conditionally MSSP) will fail, in terms of the Hodge numbers of $X$; for $X$ a curve, we will give some evidence for this claim.  

\begin{defn}[Hodge polygon]
Let $X/\mathbb{C}$ be a smooth proper variety.  Then the $m$-dimensional Hodge polygon of $X$ is the graph of the unique continuous piecewise linear function $\ell_{X, m}$ on $[0, \sum_{i+j=m} h^{i,j}]$ whose slope on $$[h^{0, m}+h^{1, m-1}+\cdots+h^{j-1, m-j+1}, h^{0, m}+h^{1, m-1}+\cdots+h^{j, m-j}]$$ is $j$, and which satisfies $\ell_{X,m}(0)=0$.
\end{defn}
For example, the $1$-dimensional Hodge polygon of a smooth proper curve of genus $g$ has slope $0$ on $[0, g]$ and slope $1$ on $[g, 2g]$.

If $X$ is defined over a finite field $\mathbb{F}_q$ and admits a lift $X^0$ to characteristic zero, there is a well-known relationship between the zeta function $\zeta_X(t)$ appearing in the Weil conjectures, and the Hodge polygons of $X^0$.  Namely, $\zeta_X(t)$ is an alternating product of polynomials $p_i(t)$, where $p_i(t)$ is the determinant of $(I-\on{Frob}_qt)$ acting on the $i$-th $\ell$-adic cohomology group of $X_{\overline{\mbb{F}_q}}$ with constant coefficients.  Dwork \cite{dworkpolygon}, Mazur \cite{mazurpolygon}, and Ogus \cite{oguspolygon} have proven that the Newton polygon of $p_m(t)$ (with respect to the valuation at $q$) lies above the $m$-dimensional Hodge polygon of $X^0$.

We suggest an imprecise analogy with motivic zeta functions, which explains the failure of the convergence of the symmetric powers of $X$.  Before expanding on this analogy, we give a precise version for curves.

Throughout this section, we will use the fact that if $X$ is a smooth projective curve with a rational point, $\on{Pic}^d(X)$ admits a stratification by locally closed subvarieties so that the natural map $\on{Sym}^d(X)\to \on{Pic}^d(X)$ (sending a divisor to its associated line bundle) is a projective space bundle over each stratum.  Furthermore, the fiber over a point $\mcl{L}$ is naturally identified with $\mbb{P}\Gamma(X, \mcl{L})$.

If $f(t)\in K_0(\on{Var}_k)[[t]]$ is a power series over $K_0(\on{Var}_k)$, we will use $[t^k]f(t)$ to denote the coefficient of $t^k$ in $f$.

It is well known (\cite[Theorem 1.1.9]{kapranov}, \cite[Theorem 3.7]{rationalitycriteria}, \cite[Theorem 7.33]{mustata}), that if $X$ is a smooth projective curve with a rational point, $$(1-t)(1-\mbb{L}t)Z^{mot}_X(t)$$ is a polynomial in $K_0(\on{Var}_k)[t]$ of degree $2g$; let us compute its coefficients.  We have that 
\begin{align*}
(1-t)(1-\mbb{L}t)Z^{mot}_X(t) &= (1-\mbb{L}t)\sum_{n=0}^\infty ([\on{Sym}^n(X)]-[\on{Sym}^{n-1}(X)])t^n\\
&=\sum_{n=0}^\infty ([\on{Sym}^n(X)]-[\on{Sym}^{n-1}(X)] +\mbb{L}(-[\on{Sym}^{n-1}(X)]+[\on{Sym}^{n-2}(X)])t^n.
\end{align*}
Here we take $[\on{Sym}^n(X)]=0$ for $n<0$.  In particular, $[t^0](1-t)(1-\mbb{L}t)Z^{mot}_X(t)=1$.  Furthermore $$[\on{Sym}^{2g}(X)]=[\on{Jac}(X)][\mbb{P}^g]$$
$$[\on{Sym}^{2g-1}(X)]=[\on{Jac}(X)][\mbb{P}^{g-1}]$$
$$[\on{Sym}^{2g-2}(X)]=[\on{Jac}(X)][\mbb{P}^{g-2}]+\mbb{L}^{g-1}$$ 
where the last equality follows from the fact that $\omega_X$ is the unique degree $2g-2$ line bundle $\mcl{L}$ with $h^0(X, \mcl{L})=g$ (from Serre Duality and Riemann-Roch), and all other line bundles $\mcl{L}'$ of degree $2g-2$ satisfy $h^0(X, \mcl{L}')=g-1$.

So
\begin{align*}
[t^{2g}](1-t)(1-\mbb{L}t)Z^{mot}_X(t)&= [\on{Jac}(X)][\mbb{P}^g]-[\on{Jac}(X)][\mbb{P}^{g-1}]+\mbb{L}(-[\on{Jac}(X)][\mbb{P}^{g-1}]+[\on{Jac}(X)][\mbb{P}^{g-2}]+\mbb{L}^{g-1})\\
&=[\on{Jac}(X)]\mbb{L}^g+\mbb{L}(-[\on{Jac}(X)]\mbb{L}^{g-1}+\mbb{L}^{g-1})\\
&=\mbb{L}^g.
\end{align*}
Furthermore, we claim that $[t^g](1-t)(1-\mbb{L}t)Z^{mot}_X(t)$ is not in $(\mbb{L})$.  For this we need to understand a bit about the stable birational geometry of $\on{Sym}^{g-1}(X)$ and $\on{Sym}^g(X)$.
\begin{lem} \label{symsblemma}
$\on{Sym}^{g-1}(X)$ and $\on{Sym}^g(X)$ are not stably birational to one another.
\end{lem}
\begin{proof}
$\on{Sym}^g(X)$ is birational to $\on{Jac}(X)$, so it suffices to show that $\on{Sym}^{g-1}(X)$ is not stably birational to $\on{Jac}(X)$.  Suppose to the contrary that there is a rational, birational map $\on{Sym}^{g-1}(X)\times\mbb{A}^n\to \on{Jac}(X)\times \mbb{A}^{n-1}$.  Then in particular some open subset $U$ of $\mbb{A}^n$ maps injectively to $\on{Jac}(X)\times \mbb{A}^{n-1}$; but as $\on{Jac}(X)$ is an Abelian variety, the image of this map under the projection to $\on{Jac}(X)$ is trivial.  Thus $U$ is entirely contained in some fiber of the projection, so $U$ maps injectively to $\mbb{A}^{n-1}$.  But $U$ has dimension $n$, so this is impossible.
\end{proof}
\begin{rem} \label{stablebiratrem}
This argument shows in general that if $X, Y$ are varieties with $\on{dim}(X)>\on{dim}(Y)$, and no rational curves pass through a general point of $X$, then $X$ and $Y$ are not stably birational.  In particular, for $C$ a smooth proper curve of genus $g$, and $n\leq g$, $\on{Sym}^{n-1}(C)$ is not stably birational to $\on{Sym}^n(C)$ (no rational curves pass through a general point of $\on{Sym}^n(C)$ as it is birational to a subvariety of $\on{Jac}(C)$).
\end{rem}
\begin{cor} \label{gcoefficientcor}
$[t^g](1-t)(1-\mbb{L}t)Z^{mot}_X(t)$ is not in $(\mbb{L})$.
\end{cor}
\begin{proof}
We have that $$[t^g](1-t)(1-\mbb{L}t)Z^{mot}_X(t)=[\on{Sym}^{g}(X)]-[\on{Sym}^{g-1}(X)]\bmod \mbb{L}.$$  As $\on{Sym}^{g-1}(X), \on{Sym}^g(X)$ are smooth and projective, Lemma \ref{symsblemma} gives that the image of this difference in $\mbb{Z}[SB]$ via $sb$ is non-zero, so we have the claim.
\end{proof}
\begin{cor}[The Newton polygon of $(1-t)(1-\mbb{L}t)Z^{mot}_X(t)$ lies below the $1$-dimensional Hodge polygon of $X$]  \label{newtonbelowhodgecor}
For $x\in K_0(\on{Var}_k)$, define $v_\mbb{L}(x)\in \mathbb{Z}_{\geq 0}\cup \{\infty\}$ to be the greatest integer $n$ so that $x\in (\mbb{L}^n)$.  Define the $\mbb{L}$-adic Newton polygon of a polynomial $p(t)=\sum a_i t^i\in K_0(\on{Var}_k)[t]$ to be the lower convex hull of the set of points $(i, v_\mbb{L}(a_i))$ (when there is no risk of confusion, we will drop the modifier ``~$\mbb{L}$-adic").  Then the Newton polygon of $(1-t)(1-\mbb{L}t)Z^{mot}_X(t)$ lies below the $1$-dimensional Hodge polygon of $X$.
\end{cor}
\begin{proof}
Recall that the $1$-dimensional Hodge polygon of $X$ consists of the segments $[(0, 0), (g, 0)]$ and $[(g, 0), (2g, g)]$.  Thus it suffices to show that $v_\mbb{L}(1)=0$, $v_\mbb{L}([t^g](1-t)(1-\mbb{L}t)Z^{mot}_X(t))=0$ and $v_\mbb{L}(\mbb{L}^{g})=g$.  

That $v_\mbb{L}(1)=0$ is simply the statement that $\mbb{L}$ is not a unit; to see this, note that it suffices to produce any homomorphism out of $K_0(\on{Var}_k)$ sending $\mbb{L}$ to a non-unit.  The homomorphism $K_0(\on{Var}_k)\to \mbb{Z}[t]$ sending a smooth proper variety to its Poincar\'e polynomial \cite[p.331]{Cut-and-paste} suffices.

Corollary \ref{gcoefficientcor} is exactly the statement that $v_\mbb{L}([t^g](1-t)(1-\mbb{L}t)Z^{mot}_X(t))=0$.

Finally $v_\mbb{L}(\mbb{L}^g)$ is greater than or equal to $g$ by definition.  But the Poincar\'e polynomial of $\mbb{L}^g$ is $t^{2g}$; thus $\mbb{L}^{g}$ is not in $(\mbb{L}^{g+1})$ (which maps to $(t^{2g+2})$ via the Poincar\'e polynomial), as desired.
\end{proof}
In fact, we claim that the Newton polygon of $(1-t)(1-\mbb{L}t)Z^{mot}_X(t)$ is actually equal to the $1$-dimensional Hodge polygon of $X$.  To see this, it suffices to show that $$v_\mbb{L}([t^n](1-t)(1-\mbb{L}t)Z^{mot}_X(t))\geq n-g$$ for $n>g$.
\begin{lem}[The Newton polygon of $(1-t)(1-\mbb{L}t)Z_X^{mot}(t)$ lies above the $1$-dimensional Hodge polygon of $X$] \label{newtonabovehodgelemma}
For $n>g,$ $[t^n](1-t)(1-\mbb{L}t)Z^{mot}_X(t)\in (\mbb{L}^{n-g}).$
\end{lem}
\begin{proof}
Let $x\in X$ be a rational point.  Let $J_{a, b, c}\subset \on{Pic}^n(X)$ be the locally closed subset consisting of those line bundles $\mcl{L}$ with $h^0(L)=a, h^0(L(-x))=b, h^0(L(-2x))=c$.  Now $$[\on{Sym}^d(X)]=\sum_{a,b,c} [J_{a,b,c}][\mbb{P}^{a-1}]$$
$$[\on{Sym}^{d-1}(X)]= \sum_{a,b,c} [J_{a,b,c}][\mbb{P}^{b-1}]$$
$$[\on{Sym}^{d-2}(X)]=\sum_{a,b,c} [J_{a,b,c}][\mbb{P}^{c-1}]$$
and so $$[t^n](1-t)(1-\mbb{L}t)Z^{mot}_X(t)=\sum_{a,b,c} [J_{a,b,c}]\left([\mbb{P}^{a-1}]-[\mbb{P}^{b-1}]+\mbb{L}(-[\mbb{P}^{b-1}]+[\mbb{P}^{c-1}])\right).$$
By Riemann's inequality, $$a\geq n-g+1; ~b\geq n-g; ~ c\geq n-g-1;$$ furthermore $$0\leq a-b, b-c\leq 1.$$  Now if $a=b=c$, the corresponding term in the sum above vanishes.  If $a>b$, (that is, $a=b+1$) and $b=c$, $$[\mbb{P}^{a-1}]-[\mbb{P}^{b-1}]=\mbb{L}^{a-1};$$ $a-1=b\geq n-g$, as desired, so the term 
\begin{align*}
[J_{a,b,c}]([\mbb{P}^{a-1}]-[\mbb{P}^{b-1}]+\mbb{L}(-[\mbb{P}^{b-1}]+[\mbb{P}^{c-1}])) &= [J_{a,b,c}]\mbb{L}^b\in (\mbb{L}^{n-g}).
\end{align*}
  An identical argument works for the cases $a=b>c$ or $a>b>c$.  Thus $$[t^n](1-t)(1-\mbb{L}t)Z^{mot}_X(t)$$ is a sum of terms in $(\mbb{L}^{n-g})$ and so is in $(\mbb{L}^{n-g})$ itself.
\end{proof}
\begin{rem} \label{newtonpolysarbitrarychar}
This argument works for curves with a rational point in arbitrary characteristic, and implies the classical result that the $q$-Newton polygons of the zeta function $\zeta_X$ associated to a smooth projective curve $X$ over a finite field $\mbb{F}_q$ lie above its associated Hodge polygons.  In particular, the $q$-Newton polygon of $\zeta_X(t)=\psi_q(Z_X^{mot}(t))$ lies above the $\mbb{L}$-adic Newton polygon of $Z_X^{mot}(t)$ (because $\psi_q(\mathbb{L})=q$), which lies above the Hodge polygon of $X$.
\end{rem}
\begin{cor}[The Newton polygon equals the Hodge polygon of $X$]  \label{newtonequalshodgecor}
Over algebraically closed fields of characteristic zero, the Newton polygon of $(1-t)(1-\mbb{L}t)Z^{mot}_X(t)$ consists precisely of the segments  $[(0, 0), (g, 0)]$ and $[(g, 0), (2g, g)]$; that is, the Hodge polygon of $X$.
\end{cor}
\begin{proof}
The follows immediately from Corollary \ref{newtonbelowhodgecor} and Lemma \ref{newtonabovehodgelemma}.
\end{proof}
Let us consider the (heuristic) relationship between Newton polygons, MSSP, and False Claims \ref{False Claim 1} and \ref{False Claim 2}.  The important identity here is the following:
\begin{equation}\label{classnumber}
\lim_{n\to \infty} [\on{Sym}^n(X)]= (1-t)Z_X^{mot}(t)|_{t=1}
\end{equation}
where both sides are evaluated in $R=\widehat{K_0(\on{Var}_k)}$, the $\mbb{L}$-adic completion of $K_0(\on{Var}_k)$ (or really any completion of $K_0(\on{Var}_k)$).

Imagine for a second that $Z_X^{mot}(t)$ was a rational or meromorphic function.  (We will leave these terms purposefully vague.)  Then one would expect the identity (\ref{classnumber}) to hold only if the power series expansion $$(1-t)Z_X^{mot}(t)=(1-t)\sum_{n=0}^\infty [\on{Sym}^n(X)]t^n$$ is valid near $t=1$; in particular, if this identity holds, $Z_X^{mot}(t)$ has no poles $y$ with $v_\mbb{L}(y)<1$.  The valuations of such poles are determined by the Newton polygon of the denominator of $Z^{mot}_X(t)$ --- one expects there to be a pole of valuation less than $1$ if and only if the Newton polygon of the denominator has a segment with slope less than $1$.

Now, by analogy with the Weil conjectures, one expects the denominator of $Z^{mot}_X(t)$ to factor as a product of polynomials relating to the even-degree cohomology of $X$; by analogy to Corollary \ref{newtonequalshodgecor}, the Newton polygons of these polynomials should lie below the $2m$-dimensional Hodge polygons of $X$.  These Hodge polygons contain a segment of slope less than $1$ (namely, of slope zero) if and only if $$h^0(X, \Omega^{2m}_X)\not=0$$ for some $m$.  

\begin{question}
Thus, we ask:  does False Claim \ref{False Claim 1} fail for all $X$ with $h^0(X, \Omega^{2m}_X)\not=0$?  We have proven this for $X$ of dimension $2$.  
\end{question}

While this paper was in preparation, Melanie Wood informed the author \cite{Melaniepersonal} of a stronger result than Theorem \ref{msspfailsthm}.

\begin{prop} \label{melanieprop}
Let $X$ be smooth and projective.  Suppose that either $\on{dim}(X)$ is even and $h^0(X, \omega_X^n)$ is non-zero, or that $h^0(X, \omega_X^{2n})$ is non-zero for some $n$.  Then if $\on{Sym}^n(X)$ is stably birational to $\on{Sym}^m(X)$, $m=n$.
\end{prop}

Applying the arguments of Corollaries \ref{falseclaimsarefalsecor} and \ref{msspfalsecor} where $X$ is a smooth surface with a non-vanishing plurigenus gives further examples of failures of False Claim \ref{False Claim 2}, and conditional failures of MSSP.  These examples include those given here, as well as e.g.~Enriques surfaces and surfaces of general type with vanishing geometric genus.  This suggests that our heuristic is incomplete. 

Unfortunately, Proposition \ref{melanieprop} does not give examples in higher dimensions, as Corollaries \ref{falseclaimsarefalsecor} and \ref{msspfalsecor} use the existence of a desingularization $S$ of $\on{Sym}^n(X)$ which satisfies $$S=\on{Sym}^n(X) \bmod \mbb{L};$$ the existence of such a desingularization in dimension greater than $2$ appears to be an open question \cite[Question 6.7]{rationalitycriteria}.  Ulyanov \cite[Theorem 2]{ulyanov} constructs a compactification of the configuration space of \emph{distinct, labeled} points in $X$ upon which the symmetric group acts with Abelian stabilizers; he remarks that desingularizing this compactification should be doable via existing methods.  So perhaps this question is tractable.  Ulyanov's compactification is a modification of the well-known Fulton-MacPherson compactification \cite{fultonmacpherson}.


\begin{thebibliography}{9}

\bibitem{Bittner}
Franziska Bittner,
The Universal Euler Characteristic for Varieties of Characteristic Zero.
Compositio Mathematica, Volume 140, Issue 04, July 2004, pp. 1011-1032.
DOI: 10.1112/S0010437X03000617

\bibitem{denefloeser}
Jan Denef; Fran\c{c}ois Loeser,
On Some Rational Generating Series Occuring in Arithmetic Geometry.
\emph{Geometric Aspects of Dwork Theory.}  Vol. I, II, de Gruyter, Berlin, 2004, pp.509-526.

\bibitem{doldthom}
Albrecht Dold; Ren\'e Thom, 
Quasifaserungen und Unendliche Symmetrische Produckte,
Ann. of Math. (2) 67 (1958), 239-281.

\bibitem{dworkpolygon}
Bernard Dwork,
A Deformation Theory for the Zeta Function of a Hypersurface.
Proc. Internat. Congr. Mathematicians (Stockholm, 1962), pp.247-259 Inst. Mittag-Leffler, Djursholm

\bibitem{FGA explained}
Barbara Fantechi; Lothar G\"ottsche; Luc Illusie; Steven L. Kleiman; Nitin Nitsure; Angelo Vistoli, 
\emph{Fundamental Algebraic Geometry.  Grothendieck's FGA Explained}.
Mathematical Surveys and Monographs 123, Amer. Math. Soc. 2005.  MR2007f:14001.

\bibitem{fultonmacpherson}
William Fulton; Robert MacPherson,
A Compactification of Configuration Spaces.
Ann. of Math., 139 (1994), 183-225.

\bibitem{gottsche}
Lothar G\"ottsche,
On the Motive of the Hilbert Scheme of Points on a Surface.
Mathematical Research Letters 8, 613-627 (2001).

\bibitem{kapranov}
Mikhail Kapranov,
The Elliptic Curve in the $S$-duality Theory and Eisenstein Series for Kac-Moody Groups.
MSRI Preprint 2000-006.
arXiv:math/0001005v2

\bibitem{kontsevich}
Maxim Kontsevich,
String Cohomology.
Lecture at Orsay, December 7, 1995.

\bibitem{lamy/sebag}
St\'ephane Lamy; Julien Sebag,
Birational Self-Maps and Piecewise Algebraic Geometry.
arXiv:1112:5706v1

\bibitem{motivicmeasures}
Michael Larsen; Valery A. Lunts,
Motivic Measures and Stable Birational Geometry,
Mosc. Math. J., 3:1 (2003) 85-95

\bibitem{rationalitycriteria}
Michael Larsen; Valery A. Lunts,
Rationality Criteria for Motivic Zeta Functions.
Compositio Mathematica, Volume 140, Issue 06, November 2004, pp. 1537-1560
DOI: 10.1112/S0010437X04000764

\bibitem{Cut-and-paste}
Qing Liu; Julien Sebag,
The Grothendieck ring of varieties and piecewise isomorphisms.
Math. Z. (2010) 265:321-342.
DOI 10.1007/s00209-009-0518-7

\bibitem{mazurpolygon}
Barry Mazur,
Frobenius and the Hodge Filtration (estimates).  
Ann. of Math. (2) 98 (1973), 58-95.

\bibitem{mumford}
David Mumford,
Rational Equivalence of $0$-Cycles on Surfaces.
J. Math. Kyoto Univ. Volume 9, Number 2 (1969), 195-204.

\bibitem{mustata}
Mircea Mustata,
\emph{Zeta Functions in Algebraic Geometry}.
\url{http://www.math.lsa.umich.edu/~mmustata/zeta_book.pdf}

\bibitem{oguspolygon}
Arthur Ogus,
Frobenius and the Hodge Spectral Sequence. 
Adv. Math. 162 (2001), no. 2, 141-172.

\bibitem{Groth}
  Jean-Pierre Serre,
  \emph{Grothendieck-Serre Correspondence}.
  American Mathematical Soc., 2004.

\bibitem{ulyanov}
Alexander P. Ulyanov,
Polydiagonal Compactifications of Configuration Spaces.
J. Algebraic Geom. 11 (2002), 129-159.

\bibitem{Ravi/Melanie}
  Ravi Vakil; Melanie Matchett Wood,
 {Discriminants in the Grothendieck Ring}.
  arXiv:1208.3166v1

\bibitem{voisin}
Claire Voisin,
\emph{Hodge Theory and Complex Algebraic Geometry II}, Chapter 10.
Cambridge Studies in Advanced Mathematics, 77, Cambridge University Press (2003)

\bibitem{Melaniepersonal}
Melanie Matchett Wood,
Personal Communication.
April 5, 2012

\end{thebibliography}
\end{document}